\documentclass[10pt]{amsart}

\usepackage{amsmath,amssymb,amsthm,amscd,amsfonts}
\usepackage{fullpage}
\usepackage{graphicx}
\usepackage{stmaryrd}
\usepackage{float}
\usepackage{hyperref}

\usepackage{algorithm}							
\usepackage[noend]{algpseudocode}					
\usepackage{caption}
\usepackage[margin=1in]{geometry}
\usepackage{url}
\usepackage{xcolor}

\usepackage{mathtools}

\usepackage{amssymb,latexsym}
\usepackage{amsmath,amsfonts,amsthm,amscd,amsxtra,enumerate}
\usepackage{mathtools}
\usepackage{caption}
\usepackage[all]{xy}
\usepackage{multicol}
\usepackage[margin=1in]{geometry}
\usepackage{epsfig}
\usepackage{pgf,tikz}

\usepackage[utf8]{inputenc}

\DeclareMathOperator{\ALPH}{Alph}

\newtheorem{conjecture}{ Conjecture}[section]
\newtheorem{theorem}[conjecture]{ Theorem}
\newtheorem{lemma}[conjecture]{ Lemma}

\newtheorem{proposition}[conjecture]{ Proposition}
\theoremstyle{definition}
\newtheorem{remark}[conjecture]{ Remark}

\newtheorem{definition}[conjecture]{ Definition}

\newtheorem{example}[conjecture]{ Example}
\allowdisplaybreaks

\begin{document}

\title{Rich Words in the Block Reversal of a Word}

\author{Kalpana Mahalingam, Anuran Maity, Palak Pandoh}

\address
	{Department of Mathematics,\\ 
	 Indian Institute of Technology Madras, 
	  Chennai, 600036, India}
	  \email{kmahalingam@iitm.ac.in, anuran.maity@gmail.com, palakpandohiitmadras@gmail.com}





	

%
%


\maketitle


\begin{abstract}
The block reversal of a word $w$, denoted by $\mathtt{BR}(w)$, is a generalization of the concept of the reversal of a word, obtained by concatenating the  blocks of the word in the reverse order. We characterize non-binary and binary words whose block reversal contains only rich words. We prove that for a binary word $w$, richness of all elements of $\mathtt{BR}(w)$  depends on  $l(w)$, the length of the run sequence of $w$. We show that if all elements of $\mathtt{BR}(w)$ are rich, then  $2\leq l(w)\leq 8$.  We also provide the structure of such words. 
\end{abstract}



\section{Introduction}
Inversions, insertions, deletions, duplications, substitutions and  translocations are some of the  operations that transform a DNA sequence from a primitive sequence (see \cite{Cantone2013,Cantone2010,Zhong2004,mahalingam2020}).
A rearrangement of chromosomes can happen when a single sequence undergoes breakage and one or more segments of the chromosome are shifted by some form of dislocation (\cite{Zhong2004}). 
Mahalingam et al. (\cite{blore}) defined the block reversal of a word which is a rearrangement of strings when dislocations happen through inversions. The authors generalized the concept of the reversal of a word where in place of reversing individual letters, they decomposed the word into factors or blocks  and considered the new word such that the blocks appear in the reverse order. 
The block reversal operation of a word $w$, denoted by $\mathtt{BR}(w)$,  is represented in Figure \ref{f2}.
\begin{figure}[h]    
\centering
 \tikzset{every picture/.style={line width=0.75pt}} 
\begin{tikzpicture}[x=0.75pt,y=0.75pt,yscale=-1,xscale=1]
\draw    (101,438.65) -- (423.79,438.65) ;
\draw    (102.51,485.23) -- (425.3,485.23) ;
\draw   (102.51,485.23) -- (128.36,485.23) -- (128.36,490.41) -- (102.51,490.41) -- cycle ;
\draw   (385.88,485.23) -- (425.3,485.23) -- (425.3,490.41) -- (385.88,490.41) -- cycle ;
\draw   (101,438.65) -- (140.42,438.65) -- (140.42,443.83) -- (101,443.83) -- cycle ;
\draw   (332.9,485.23) -- (385.88,485.23) -- (385.88,490.41) -- (332.9,490.41) -- cycle ;
\draw   (140.42,438.65) -- (193.4,438.65) -- (193.4,443.83) -- (140.42,443.83) -- cycle ;
\draw   (397.94,438.65) -- (423.79,438.65) -- (423.79,443.83) -- (397.94,443.83) -- cycle ;
\draw   (128.36,485.23) -- (173.05,485.23) -- (173.05,490.41) -- (128.36,490.41) -- cycle ;
\draw   (353.25,438.65) -- (397.94,438.65) -- (397.94,443.83) -- (353.25,443.83) -- cycle ;
\draw   (232.02,485.23) -- (263.9,485.23) -- (263.9,490.41) -- (232.02,490.41) -- cycle ;
\draw   (250.98,438.78) -- (282.86,438.78) -- (282.86,443.96) -- (250.98,443.96) -- cycle ;

\draw (106.33,420.03) node [anchor=north west][inner sep=0.75pt]    {$B_{1}$};
\draw (147.79,419.63) node [anchor=north west][inner sep=0.75pt]    {$B_{2}$};
\draw (398.59,419.63) node [anchor=north west][inner sep=0.75pt]    {$B_{k}$};
\draw (256.56,419.14) node [anchor=north west][inner sep=0.75pt]    {$B_{i}$};
\draw (359.47,418.09) node [anchor=north west][inner sep=0.75pt]    {$B_{k-1}$};
\draw (102.51,490.41) node [anchor=north west][inner sep=0.75pt]    {$B_{k}$};
\draw (134.88,490.59) node [anchor=north west][inner sep=0.75pt]    {$B_{k-1}$};
\draw (237.72,491.04) node [anchor=north west][inner sep=0.75pt]    {$B_{i}$};
\draw (346.75,492.69) node [anchor=north west][inner sep=0.75pt]    {$B_{2}$};
\draw (391.97,492.04) node [anchor=north west][inner sep=0.75pt]    {$B_{1}$};
\draw (256.94,448.95) node [anchor=north west][inner sep=0.75pt]    {$w$};
\draw (251.91,502.46) node [anchor=north west][inner sep=0.75pt]    {$w'$};
\end{tikzpicture} \caption{For $w \in \Sigma^*$, $w'\in \mathtt{BR}(w)$}
    \label{f2}
\end{figure}
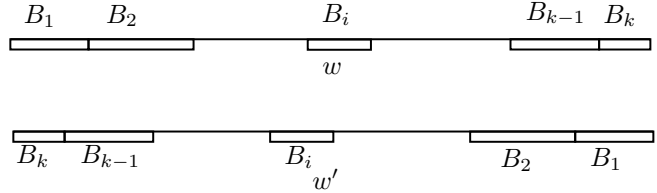
If the word $w$ can be expressed as a concatenation of its  factors or blocks $B_i$ such that $w=B_1B_2\cdots B_k$, then $w'=B_kB_{k-1}\cdots B_1$  is an element of $\mathtt{BR}(w)$. Since there are multiple ways to divide a word into blocks, the block reversal of a word forms a set.

Mahalingam et al. (\cite{blore}) proved that there is a strong connection between the block reversal  and the non-overlapping inversion of a word. A non-overlapping inversion of a word is a set of inversions that do not overlap with each other. 
In 1992, Sch\"oniger et al. (\cite{Schon1992}) presented a heuristic for computing the edit distance when non-overlapping inversions are allowed.
They presented an $\mathcal{O}$($n^
6$) exact solution for the
alignment with the non-overlapping inversion problem and showed the non-overlapping inversion operation ensures that all inversions occur in one mutation step. Instances of problems considering the non-overlapping inversions include the string alignment problem,  the edit distance problem, the approximate matching problem, etc. (\cite{Cantone2010,Cantone20,Kece93,Augusto2006}). Kim et al. (\cite{Kim2015}) studied the non-overlapping inversion on strings from a formal language theoretic approach.

A word is a palindrome if it is equal to its reverse. Let $|w|$ be the  length of the word $w$. It was proved by Droubay et al.  (\cite{epi}) that a word $w$ has at most $|w|$ non-empty distinct palindromic factors. The words that achieve the bound were referred to as rich words by Glen et al. (\cite{palrich}). Several properties of rich words were studied in the literature (see \cite{pal3,epi,palrich,guo}). Droubay et al. (\cite{epi}) proved that a word $w$  contains exactly $|w|$ non-empty distinct palindromic factors iff the longest
palindromic suffix of any prefix $p$ of $w$ occurs exactly once in $p$. Guo et al. (\cite{guo}) provided necessary and sufficient conditions for richness in terms of the run-length encoding of binary words. It is known that on a binary alphabet, the set of rich words contain factors of the period-doubling words, factors of Sturmian words, factors of complementary symmetric Rote words, etc. (see \cite{palrot1, epi, schaclo}). In a non-binary alphabet, the set of rich words contain, for example factors of Arnoux–Rauzy words and factors of words coding symmetric interval exchange.

There are many results in the literature regarding the occurrence of rich words in infinite and finite words, but there are significantly fewer results about the occurrence of rich words in a language. The occurrence of rich words in the conjugacy class of a word $w$, denoted by $C(w)$, is a well-studied concept in literature (see \cite{careymusic, palrich,  restivobwt, oeis1}).
 Shallit et al. (\cite{oeis1}) calculated the number of binary words $w$ of a particular length such that every conjugate of $w$ is rich.
A word $w$ is said to be circularly rich if all of the conjugates of $w$ (including itself) are rich, and $w$ is a product of two palindromes.
 Glen et al.  (\cite{palrich}) studied circularly rich words and proved the equivalence conditions for circularly rich words. They proved that a  word $w$ is circularly rich iff the infinite word $w^\omega$ is rich iff $ww$ is rich where $w^\omega$ is a word formed by concatenating infinite copies of $w$.
 Restivo et al. (\cite{restivobwt, bwtRESTIVO}) outlined relationships between circularly rich words and the Burrows–Wheeler transform, a highly efficient data compression algorithm.

 In many musical contexts, scale and rhythmic patterns are extended beyond a single iteration of the interval of periodicity. From the equivalent conditions for circularly rich words, proved by Glen et al. (\cite{palrich}), if $w$ is the step pattern of an octave-based scale and is rich, and if $ww$ is also rich, then the property can be extended without limit $(w^\omega)$. Lopez et al. (\cite{lopezmusic}) examined that circular palindromic richness is inherent in numerous musical contexts, including all well-formed and maximally even sets and also in non well-formed scales which display three different step sizes. 
Carey (\cite{careymusic}) also deeply studied circularly rich words from a music theory perspective. He proposed that perfectly balanced scales that display circular palindromic richness and also exhibit relatively few step differences may prove to be advantageous from a cognitive and musical perspective. Since block reversal operation is a generalization of conjugate operation, the study of rich words in the block reversal of the word has  possible applications in music theory and data compression techniques.

In this paper, we characterize words whose block reversal contains only rich words. 
We find a necessary and sufficient condition for a non-binary word such that all elements in its block reversal are rich.
For a binary word $w'$, we prove that the richness of elements of $\mathtt{BR}(w')$
 depends on $l(w')$ which is  the length of the run sequence of $w'$. We show that for a binary word $w$, if  all elements of $\mathtt{BR}(w)$ are rich, then  $2\leq l(w) \leq 8$. We also find the  structure of binary words whose block reversal consists of only rich words.
The paper is organized as follows.
In Section \ref{sec3}, we prove that  for a non-binary word $w$, all elements of $\mathtt{BR}(w)$ are rich iff $w$ is either of the form $a_1a_2a_3 \cdots a_k$ or $a_j^{|w|}$ where each $a_i\in \Sigma$ is distinct.
 In Section \ref{sec4}, we  show that for a  binary word $w$, all elements of $\mathtt{BR}(w)$ are rich if $l(w) =2$.  We also show that if all elements of $\mathtt{BR}(w)$ are rich, then  $2\leq l(w) \leq 8$. We discuss the case when $3\leq l(w)\leq 8$ separately in detail and provide the structure of words such that all elements in their block reversal are rich. We end the paper with a few concluding remarks.

\section{Basic definitions and notations}\label{sec2}
Let $\Sigma$ be a non-empty set of letters. A word $w=[a_{i}]$ over $\Sigma$ is a finite sequence of letters from $\Sigma$ where $a_i$ is the $i$-$th$ letter of $w$. We denote the empty word by $\lambda$. By $\Sigma^*$, we denote the set of all words over $\Sigma$ and  $\Sigma^+=\Sigma^*\setminus \{\lambda\}$. The length of a word $w$, denoted by $|w|$, is the number of letters in $w$.  $\Sigma^n$ and $\Sigma^{\geq n}$ denote the set of all words of length $n$ and the set of all words of length greater than or equal to $n$, respectively. For $a \in \Sigma$, $|w|_a$ denotes the number of occurrences of $a$ in $w$. A word $u$ is a factor or block of the word $w$ if $w=puq$ for some $p, q\in \Sigma^*$. If $p= \lambda$, then $u$ is a prefix of $w$ and if $q = \lambda$, then $u$ is a suffix of $w$.  Let $Fac(w)$ denote the set of all  factors of the word $w$. $\ALPH(w)$ denotes the set of all letters in $w$. Two words $u$ and $v$ are called conjugates of each other if there exist $x,y\in \Sigma^*$ such that $u=xy$ and $v=yx$.
For a word $w= w_1w_2\cdots w_n$ such that $w_i\in \Sigma$, the reversal of $w$, denoted by $w^R$, is the word $w_n\cdots w_2w_1$. A word $w$ is a palindrome if $w=w^R$. By $P(w)$, we denote the number of  all non-empty palindromic factors of $w$. A word $w$ has at most $|w|$ distinct non-empty palindromic factors. The words that achieve the bound are called rich words. 

Every non-empty word $w$ over $\Sigma$ has a unique encoding of the form $w = a_1^{n_1}a_2^{n_2}\ldots a_k^{n_k},$ where $n_i \geq 1$, $a_i \neq a_{i+1}$ and $a_i\in \Sigma$ for all $i$. This encoding is called run-length encoding of $w$ (\cite{guo}). The word $a_1 a_2 \ldots a_k$ is called the trace of $w$. The sequence $(n_1, n_2, \ldots, n_k)$ is called the run sequence of $w$ and the length of the run sequence of $w$ is $k$.  For any binary word $w$ over $\Sigma = \{a,b\}$, the complement of $w$, denoted by $w^c$, is the word  $\phi(w)$ where, $\phi$ is a morphism such that $\phi(a)=b$ and $\phi(b)=a$. For example, if $w= ababb$, then $w^c = babaa$.
We recall the definition of the  block reversal of a word from Mahalingam et al. (\cite{blore}).
\begin{definition}\cite{blore} \label{br}
Let $w,\;B_i \in \Sigma^+$ for all $i$. The block reversal of $w$, denoted by $\mathtt{BR}(w)$, is the set $$\mathtt{BR}(w) =\{ B_tB_{t-1}\cdots  B_1 \;:\; \;w=B_1B_2 \ldots  B_t, \; ~t\ge 1\}.$$
\end{definition}
Note that a word can be divided into a maximum of $|w|$ blocks. We illustrate Definition \ref{br}  with the help of an example.

\begin{example}\label{e2}
Let $\Sigma=\{a, b,c\}$. Consider $u=abbc$ over $\Sigma$. Then, $$\mathtt{BR}(u) = \{ cbab, cbba, cabb, bbca, bcab, abbc, bcba \}.$$
\end{example}

For more information on  words, the reader is referred to Lothaire (\cite{Lothaire1997}) and Shyr (\cite{Shyr2001}). 

\section{Block Reversal of Non-binary Words}\label{sec3}
It is well known that a rich word $w$ contains exactly $|w|$  distinct palindromic factors.
In this section, we find a necessary and sufficient condition for a non-binary word such that all elements in its block reversal are rich.

We recall the following from Glen et al. (\cite{palrich}). 
\begin{theorem}\cite{palrich}\label{tglen}
For any word $w$, the following properties are equivalent:\\
(i) $w$ is rich;\\
(ii) for any factor $u$ of $w$, if $u$ contains exactly two occurrences of a palindrome $p$ as a prefix and as a suffix only, then $u$ is itself a palindrome.
\end{theorem}
\begin{lemma}\cite{palrich}\label{rich}
If $w$ is rich, then
\begin{itemize}
    \item  all factors of $w$   are rich.
    \item  $w^R$ is  rich.
\end{itemize}
\end{lemma}

We first  give a necessary condition under which $\mathtt{BR}(w)$ contains at least one rich word.
\begin{lemma}\label{45tn}
 Let $w \in \Sigma^n$. If $\mathtt{BR}(w)$ has no rich element, then $|\ALPH(w)|< n-1$.
\end{lemma}
\begin{proof}
 Let $w \in \Sigma^n$ such that $\mathtt{BR}(w)$ contains no rich element. We prove that if $|\ALPH(w)| \geq  n-1$, then there exists at least one rich word in $\mathtt{BR}(w)$. 
If $|\ALPH(w)|=n $, then all elements in $\mathtt{BR}(w)$ are rich. 
If $|\ALPH(w)|=n-1 $, then for $u_1, u_2, u_3 \in \Sigma^*$, $w = u_1 a u_2 a u_3$ such that $a \notin \ALPH(u_i)$ for all $i$ and 
$\ALPH(u_i) \cap \ALPH(u_j) = \emptyset$ for $i\neq j$. Now, $u_3u_2a^2u_1 \in \mathtt{BR}(w)$ is a rich word.
\end{proof}
We now give  an example of a word $w \in \Sigma^n$ with $|\ALPH(w)|=n-2$ such that $\mathtt{BR}(w)$ contains no rich word.

\begin{example}\label{45tnr}
For $a, b \in \Sigma$, consider $w= u_1 \textbf{a} u_2 \textbf{b} u_3 \textbf{b} u_4 \textbf{a} u_5$ such that $a, b \notin \ALPH(u_i)$, $|u_i|\geq 3$ for each $i$, $\ALPH(u_i) \cap \ALPH(u_j)=\emptyset$ for $i\neq j$ and $\sum_{i=1}^{i=5}|\ALPH(u_i)|=|w|-4$. Then, $|\ALPH(w)|=|w|-2$.
We denote by $\pi(w)$, the set of all permutations of the word $w$, i.e., $\pi(w)=\{u\in \Sigma^*|\; |u|_a=|w|_a \text{ for all } a\in \Sigma \}$. One can easily observe that $\mathtt{BR}(w)$ is a subset of $\pi(w)$.
If $\pi(w)$ has no rich words, then $\mathtt{BR}(w)$ also has no rich words. Suppose there is a $ v \in \pi(w)$ such that $v$ is rich then, as  $|\ALPH(v)|=|v|-2$,   $|v|_a=|v|_b=2$,
 by Theorem \ref{tglen}, we have, $\{ a \alpha a, b \alpha' b~ |~  \alpha, \alpha' \in \Sigma^{\geq 2} \text{ such that } \alpha \neq bzb, \alpha' \neq az'a \text{ where } z, z' \in \Sigma \cup \{\lambda\} \}$ $\cap \; Fac(v) = \emptyset$. Otherwise, if $a \alpha a$ or $b \alpha' b$ lies in $Fac(v)$, then as $\alpha \neq bzb$, $\alpha' \neq az'a$  where  $z, z' \in \Sigma \cup \{\lambda\}$, $\alpha, \alpha' \in \Sigma^{\geq 2}$, $|\alpha|_a = 0$ and $|\alpha'|_b = 0$, we have, $a \alpha a$ and $b \alpha' b$ are not palindromes, which contradicts Theorem \ref{tglen}. 
This implies that $v$ is of  one of the following forms:
\begin{align}\label{algrt1}
    v_1 x^2 v_2 y^2 v_3,\; v_1 x x_1 x v_2 y x_2 y v_3, \; v_1 xy^2x v_2,\; v_1 xx_1x v_2 y^2 v_3, \; v_1 y^2 v_2 xx_1x v_3, \; v_1 xyxy v_2, \; v_1 x y x_1 y x v_2  
\end{align}
where $x \neq y \in \{a, b\}$, $ x_1, x_2 \in \Sigma \setminus \{a, b\}$  and $v_i \in \Sigma^*$ for all $i$.
Now, from the structure of $w$, we observe that $\mathtt{BR}(w)$ does not contain any element of forms in (\ref{algrt1}). 
Thus, $v \notin \mathtt{BR}(w)$. Now, $\mathtt{BR}(w) \subseteq \pi(w)$ and $v\in\pi(w) $ is  rich implies  no element of $\mathtt{BR}(w)$ is rich.
\end{example}
Note that some elements of $\mathtt{BR}(w)$ may not be rich even when $w$ is rich.
For example, the word $w=abbc$ is rich, but $ bcab \in \mathtt{BR}(w)$ is not rich. 
We now give a necessary and sufficient condition on a non-binary word $w$ such that all elements of $\mathtt{BR}(w)$ are rich. We need the  following results.

\begin{lemma} \label{nori}
Let  $w = a_1^{n_1}a_2uv$  where $a_1\neq a_2$, $u \in \{a_1,a_2\}^+$, $v \in (\Sigma \setminus \{ a_1, a_2\})^+$ and $n_1 \geq 1$. Then, there exists an element in $\mathtt{BR}(w)$ which is not rich. 
\end{lemma}
\begin{proof}
Let $w = a_1^{n_1} a_2 u v$ such that $a_1\neq a_2$, $u \in \{a_1,a_2\}^+$, $v \in (\Sigma \setminus \{ a_1, a_2\})^+$ and $n_1 \geq 1$. Let $u = u' a_i$
  where $i=1$ or $2$.  Then, $w = a_1^{n_1} a_2 u'a_i v $. 
 Note that $w'=u' a_i v  a_2 a_1^{n_1} \in \mathtt{BR}(w)$  and  $w''=a_i v a_1^{n_1} a_2 u' \in \mathtt{BR}(w)$.   The factor $ a_i v a_2  a_1$ of $w'$  is not a palindrome for $i=1$ as $a_i\notin \ALPH(v)$ and similarly the factor  $ a_iv a_1^{n_1}  a_2$ of $w''$ is not a palindrome for $i=2$.  Hence, by Theorem \ref{tglen} and Lemma \ref{rich}, $w'\in \mathtt{BR}(w)$ is not rich when $i=1$ and  $w''\in \mathtt{BR}(w)$ is not rich when $i=2$.

    
\end{proof}

We also have the following:
\begin{lemma}\label{notrich}
For  $u_1, u_2, u_3, u_4, u_5 \in \Sigma^*$, consider $w= u_1 a_i u_2 a_j u_3 a_{k} u_4 a_i u_5$ such that $a_j \neq a_k$, $a_j \neq a_i \neq a_k$ and  $a_k$ is not a suffix of $u_3$. Then, there exists an element in $\mathtt{BR}(w)$ which is not rich.
\end{lemma}
\begin{proof}
For  $u_1, u_2, u_3, u_4, u_5 \in \Sigma^*$, consider $w= u_1 a_i u_2 a_j u_3 a_{k} u_4 a_i u_5$ where $a_j \neq a_k$, $a_j \neq a_i \neq a_k$ and  $a_k$ is not a suffix of $u_3$. We have the following cases:
\begin{itemize}
    \item  $a_i\notin \ALPH(u_3) : $  
     Then, $w'= u_5 u_4  a_i a_k u_3  a_j a_i u_2 u_1 \in \mathtt{BR}(w)$   and $a_j \neq a_k$ implies $a_i a_k u_3  a_j a_i$ is not a palindromic factor of $w'$. Then, by Theorem \ref{tglen}, $w'$ is not rich.
    \item $a_i\in \ALPH(u_3) : $  Then, let $u_3 = u_3' a_i u_3''$ such that  $u_3', u_3'' \in \Sigma^*$ and $|u_3''|_{a_i}=0$.
    Now, $w''= u_5 u_4 a_i a_{k} u_3'' a_i u_1 a_i u_2 a_j u_3' \in \mathtt{BR}(w)$. If $w''$ is not rich, then we are done. If $w''$ is rich, then $a_i a_{k} u_3'' a_i\in Fac(w'')$ and since, $|u_3''|_{a_i}=0$, by  Theorem \ref{tglen},  $a_i a_{k} u_3'' a_i$ is a palindrome. This implies $u_3''=\lambda$ as $a_k$ is not a suffix of $u_3$. Then, $w=u_1 a_i u_2 a_j u_3' a_i a_{k} u_4 a_i u_5$. Now, $w'''= u_4 a_i u_5  u_3' a_i a_{k}  a_j a_i u_2 u_1 \in \mathtt{BR}(w) $. Then, $a_i a_{k}  a_j a_i \in Fac(w''')$ is not a palindrome  as $a_j\neq a_k$. Therefore, by Theorem \ref{tglen}, $w'''$ is not rich.
\end{itemize}
\end{proof}

Now, we find a necessary and sufficient condition for a non-binary word such that all elements in its block reversal are rich.

\begin{theorem}
Let $w$ be a non-binary word. Then, all elements of $\mathtt{BR}(w)$ are rich iff $w$ is either of the form $a_1a_2a_3 \cdots a_k$ or $a_i^{|w|}$ where $a_i\in \Sigma$ are distinct.
\end{theorem}
\begin{proof}

Let $w \in \Sigma^*$. If $|\ALPH(w)|=1$, we are done. Assume $|\ALPH(w)|\geq 3$ and consider the run-length encoding of $w$ to be $ a_1^{n_1}a_2^{n_2}a_3^{n_3} \cdots a_k^{n_k}$ where $a_i \neq a_{i+1}\in \Sigma$, $k \geq 3$ and $n_i \geq 1$. Let all elements of $\mathtt{BR}(w)$ be rich. We have the following cases: 
\begin{itemize}
\item All $a_t$'s are distinct for $1\leq t\leq k$ : We prove that $n_t=1$ for $1\leq t\leq k$.
Assume if possible that there exists at least one $n_j \geq 2$ for some $j$, i.e., $n_j = 2m +s$ for $m\geq 1$ and $ s\in \{0,1\}$. Let $\gamma = a_1^{n_1} a_2^{n_2} \cdots a_{j-1}^{n_{j-1}}$ and $\gamma'=a_{j+1}^{n_{j+1}} a_{j+2}^{n_{j+2}} \cdots a_k^{n_k}$.
Then, $w =  \gamma a_j^{2m+s}\gamma'$.
Since, $a_j^m \gamma'  \gamma a_j^m a_j^s \in \mathtt{BR}(w)$  is  rich and all $a_t$'s are distinct, by Theorem \ref{tglen}, $u = a_j^m \gamma' \gamma a_j^m$ is a palindrome. Now, as $|\ALPH(w)|\geq 3$ and all $a_t$'s are distinct, $u$ is not a palindrome which is a contradiction. Therefore, $n_t = 1$ for  $1\leq t\leq k$ and $w=a_1a_2a_3 \cdots a_k$.
\item Otherwise, suppose  $i$ is the least index such that $|a_1a_2\cdots a_k|_{a_i}\geq 2$ and $a_j=a_i$ where $a_l\neq a_i$  for $i+1\leq l\leq j-1$ i.e., $j$ is the first position at which $a_i$ repeats for  $i<j$. We have the following cases :
\begin{itemize}
\item  $i\geq 3$ :  Note that $a_i \neq a_1$ and $a_i \neq a_2$. Let $n_i \geq n_j$ such that $n_i = n_j + s' $ where $s' \geq 0$. Now, for $\delta = a_3^{n_3} a_4^{n_4} \cdots a_{i-2}^{n_{i-2}}$, $\alpha = a_{i+1}^{n_{i+1}} a_{i+2}^{n_{i+2}} \cdots a_{j-1}^{n_{j-1}}$ and $\beta = a_{j+1}^{n_{j+1}} a_{j+2}^{n_{j+2}}\cdots a_{k}^{n_{k}}$, we have,
 $$w =  a_1^{n_1}a_2^{n_2} \delta a_{i-1}^{n_{i-1}} a_i^{n_j + s'} \alpha a_j^{n_j} \beta. $$
Since, $\beta a_j^{n_j} a_1^{n_1} a_2^{n_2} \delta a_{i-1}^{n_{i-1}} a_i^{n_j + s'} \alpha \in \mathtt{BR}(w)$ is rich, by Theorem \ref{tglen}, we get, $a_j^{n_j} a_1^{n_1} a_2^{n_2} \delta a_{i-1}^{n_{i-1}} a_i^{n_j}$ is a palindrome, and hence, $a_1 = a_{i-1}$. Similarly, as $\beta a_j^{n_j} a_2^{n_2} \delta a_{i-1}^{n_{i-1}} a_i^{n_j} a_i^{s'} \alpha a_1^{n_1} \in \mathtt{BR}(w)$, by Theorem \ref{tglen}, we have, $a_j^{n_j} a_2^{n_2} \delta a_{i-1}^{n_{i-1}} a_i^{n_j}$ is a palindrome, which gives $a_2 = a_{i-1}$. Thus, $a_1 = a_2$ which is a contradiction. A symmetrical argument holds for the case $n_i< n_j$.

\item  $i\leq 2$ : Let $i=1$, i.e., $a_1$ has a  repetition and there exists  an index $l>1$  such that $a_1 = a_l$. If all elements in $\mathtt{BR}(w)$ are rich, then by Lemma \ref{notrich}, there exists at most one distinct letter between $a_1$ and $a_l$. Similarly, between any two occurrences of $a_2$, there exists at most one distinct letter.  Then, $w$ can only be of forms $a_1^{n_1} a_2^{n_2} a_3^{n_3} a_2 z'$ or $ a_1^{n_1} a_2 u v$ where $a_1 \neq a_3$, $z' \in \{\Sigma \setminus \{a_1\}\}^*$, $u \in \{ a_1, a_2 \}^+$ and $v \in (\Sigma \setminus \{ a_1, a_2 \} )^+$. If $w$ is in form $a_1^{n_1} a_2^{n_2} a_3^{n_3} a_2 z'$, then $z'a_2 a_3^{n_3} a_1^{n_1} a_2^{n_2} \in \mathtt{BR}(w)$. Since, $a_2 a_3^{n_3} a_1^{n_1} a_2$ is not a palindrome, by Theorem \ref{tglen}, $z'a_2 a_3^{n_3} a_1^{n_1} a_2^{n_2} $ is not rich, a contradiction. Now, consider $w$ is in form $ a_1^{n_1} a_2 u v$.
Note that as $|\ALPH(w)|\geq 3,$ $v\neq \lambda$.
By Lemma \ref{nori}, there exists an element in $\mathtt{BR}(w)$ which is not rich, a contradiction.
 \end{itemize}
 \end{itemize}
Thus, if $|\ALPH(w)|\geq 3$ and  all elements of $\mathtt{BR}(w)$ are rich, then $w=a_1a_2a_3 \cdots a_k$ where each $a_i\in \Sigma$ are distinct.\\
The converse is  straightforward.
\end{proof}

\section{Block Reversal of Binary Words}\label{sec4}


Anisiu et al. (\cite{pcofw}) showed that any binary word of length greater than $8$, has at least $8$ non-empty palindromic factors. A set of  words that achieve the bound of having exactly $8$ palindromic factors was given by Fici et al. (\cite{lepin}). 
 We  recall the definition  of $k$-$th$ power of $u \in \Sigma^*$ from Brandenburg (\cite{FUPH}) as the prefix of least length  $u'$ of $u^n$ where $n\geq k$ such that $|u'|\geq k|u|$.  For example, given a word $aba$, the $\frac{5}{3}$-$th$ power of $aba$ is $ aba^{(\frac{5}{3})} = abaab$.  Fici et al. (\cite{lepin}) showed that for all $u\in C(v)$ where $v= abbaba$, $P(u^{(\frac{n}{6})})=8$, $n\geq 9$. Mahalingam et al. (\cite{lep})  characterized  words $w$ such that $P(w) =8$.  They proved that a binary word $w$ has $8$ palindromic factors iff $w$ is of the form $u^{(\frac{n}{6})}$ where  $u\in C(v) \cup  C(v^R)$ and $v=abbaba$.

In this section, we discuss the case  of binary words. Let $l(w)$ be the length of the run sequence of a binary word $w$. We prove that 
if $l(w)=2$, then all  elements of $\mathtt{BR}(w)$ are rich and if $l(w)\geq 9$, then there exists an element in $\mathtt{BR}(w)$ that is not  rich.  
Then, we study the block reversal of binary words with $3\leq l(w)\leq 9$. The results in this section also hold for complement words as we have considered unordered alphabet $\Sigma=\{a,b\}$.


\subsection{\textbf{Block reversal of binary words $\bf{w}$ with} $\bf{l(w)= 2}$ $\bf{\&}$ $\bf{l(w)\geq 9}$}
\vspace{.25cm}
Now, we discuss the block reversal of binary words $w$ with $l(w)=2$ $\&$ $l(w)\geq 9$.
We first recall the following from Guo et al. (\cite{guo}).
\begin{proposition} \cite{guo}\label{runlen}
Every binary word having a run sequence of length less than or equal to $4$ is rich.
\end{proposition}
It was verified by Anisiu et al.  (\cite{pcofw}) that for all  short binary words (up to $|w|=7$), $P(w)=|w|$.
We observe that for words $w$ with $|w|> 7$, some elements of $\mathtt{BR}(w)$ may not be rich even when  $w$ is rich.
For example, $w = a^2b^3a^3$  is rich but $a^2bab^2a^2 \in \mathtt{BR}(w)$, is not rich. We discuss the case when $l(w)=2$ for a word $w$ in the following.

\begin{proposition}\label{u1}
If $w = a^{n_1}b^{n_2}$ where $n_1, n_2 \geq 1$, then all  elements of $\mathtt{BR}(w)$ are rich.
\end{proposition}
\begin{proof}
Let $w=a^{n_1}b^{n_2}$ where $n_1, n_2 \geq 1$. Then, 
$$\mathtt{BR}(w) = \{b^{n_2'}a^{n_1'}b^{n_2-n_2'}a^{n_1-n_1'}\; : \;0 \leq n_1' \leq n_1, \; 0 \leq n_2' \leq n_2\}.$$
Since the length of the run sequence of each element of $\mathtt{BR}(w)$ is less than or equal to $4$, by Proposition \ref{runlen}, all  elements of $\mathtt{BR}(w)$ are rich.
\end{proof}

We now prove that for a binary word $w$ with length of the run sequence greater than  $8$, there exists an element in $\mathtt{BR}(w)$ that is not rich.
We recall the following from Mahalingam et al. (\cite{blore}).
\begin{lemma}\label{hhh}\cite{blore}
$\mathtt{BR}(v) \mathtt{BR}(u) \subseteq \mathtt{BR}(uv)$ for $u,\; v\in \Sigma^*$.
\end{lemma}

We now have the following:

\begin{proposition}\label{u3}
Let $w \in \{a, b\}^*$ such that $l(w)\geq 9$. Then, there exists an element in  $\mathtt{BR}(w)$ that is not rich.
\end{proposition}
\begin{proof}
 Let $w \in \{a, b\}^*$ such that $l(w) \geq 9$. Since, $l(w) \geq 9$, then for $n_i\geq 1$, consider  $w' = a^{n_1} b^{n_2} a^{n_3} b^{n_4} a^{n_5} b^{n_6}\\ a^{n_7} b^{n_8} a^{n_9}$ to be a prefix of $w$. If $w$ is not rich, then we are done. Otherwise, $w$ is rich, then by Lemma \ref{rich}, all  factors of $w$ are  rich.  Also, from Lemma \ref{hhh}, we have, $\mathtt{BR}(v) \mathtt{BR}(u) \subseteq {\mathtt{BR}(uv)}$ for $u,\; v\in \Sigma^*$. We show that there exists an element in  $\mathtt{BR}(w')$ that is not  rich to complete the proof. Suppose to the contrary that all  elements of $\mathtt{BR}(w')$ are rich. Let $w_1, w_2, w_3 \in \mathtt{BR}(w')$ where 
 $$w_1 = a^{n_9+n_7} b^{n_8+n_6} a b a^{n_5-1+n_3}b^{n_4-1+n_2}a^{n_1},$$
 $$w_2 = b^{n_8} a^{n_9+n_7} b a b^{n_6-1+n_4} a^{n_5-1+n_3+n_1}b^{n_2} \text{ \;\;and}$$
 $$ w_3 = a^{n_9} b^{n_8+n_6} a^{n_7+n_5} b a b^{n_4-1+n_2} a^{n_3-1+n_1}.$$ We have the following:
 \begin{itemize}
     \item  If  $n_3+n_5\geq 3$, then $a^2 b^{n_8+n_6} a b a^2$ is a factor of $w_1$ that contains exactly two occurrences of a palindrome $a^2$ as a prefix and as a suffix. By Theorem \ref{tglen}, if $w_1$ is rich, then $a^2 b^{n_8+n_6} a b a^2$ is a palindrome which is a contradiction. Hence, $n_3 = n_5 =1$.
     \item  If  $n_4+n_6\geq 3$, then $a^2 b a b^{n_6-1+n_4} a^2$ is a factor of $w_2$ that contains exactly two occurrences of a palindrome $a^2$ as a prefix and as a suffix. By Theorem \ref{tglen}, if $w_2$ is rich, then $ a^2 bab^{n_6-1+n_4} a^2$ is a palindrome which is a contradiction. Hence,  $n_4 =n_6=1$.
     \item If $n_2 \geq 2$, then 
     $b^2 a^{n_7+n_5} b a b^2$
      is a factor of $w_3$ that contains exactly two occurrences of a palindrome $b^2$ as a prefix and as a suffix. By Theorem \ref{tglen}, if $w_3$ is rich, then $ b^2 a^{n_5+n_7} b a b^2$ is a palindrome which is a contradiction. Hence,  $n_2=1$.
     
 \end{itemize}
 Hence, we have $n_2=n_3=n_4=n_5=n_6=1$ and  $w' = a^{n_1} b a b a b a^{n_7} b^{n_8} a^{n_9}$.
 Let $w_4,\; w_5\in \mathtt{BR}(w')$ where $w_4= a^{n_9+n_7} b^{n_8} a b^2 a^{n_1+1} b$ and $w_5 = a^{n_9+n_7} b^{n_8+1} a b^2 a^{1+n_1}$.  Now,
   $a^2 b^{n_8}a b^2 a^2$
      is a factor of $w_4$ that contains exactly two occurrences of a palindrome $a^2$ as a prefix and as a suffix. By Theorem \ref{tglen}, since $w_4$ is rich,  $n_8=2$. Also,   $a^2 b^{n_8+1}a b^2 a^2$
      is a factor of $w_5$ that contains exactly two occurrences of a palindrome $a^2$ as a prefix and as a suffix. By Theorem \ref{tglen}, since $w_5$ is rich, $n_8=1$, which is a contradiction. Thus, there always exists an element in $\mathtt{BR}(w')$ that is not rich. 
\end{proof}

\subsection{\bf{Block reversal of binary words $\bf{w}$ with} $\bf{3\leq l(w)\leq 8}$}
\vspace{.25cm}
We now consider the case  for a binary word $w$ such that $3\leq l(w)\leq 8$. We observe that the result varies with the structure of the word. We compile all results towards the end of this section. We first recall the following from Anisiu et al. (\cite{pcofw}).
\begin{theorem}\cite{pcofw}\label{7ric} If $w$ is a binary word of length less than $8$, then $P(w)=|w|$. If $w$ is a binary word of length $8$, then $7\leq P(w)\leq 8$ and  $P(w)= 7$ iff $w$ is of the form $aabbabaa$ or $aababbaa$. 
\end{theorem}
Now, with the help of examples, we  illustrate that all  elements of $\mathtt{BR}(w)$ may  be rich for a binary word $w$ such that $3\leq l(w)\leq 8$.

\begin{example}\label{u2}
  For $3\leq l(w)\leq 8$, consider \[w=\left \{\begin{array}{cc}
   a(ba)^i  &  \text{for $i=\frac{l(w)-1}{2}$ and $l(w)$ odd,}\\
    (ab)^i & \text{for $i=\frac{l(w)}{2}$ and $l(w)$ even.}\end{array} \right.\]
    It can be observed that all  elements of $\mathtt{BR}(w)$ are rich.

\end{example}

Thus, from Propositions \ref{u1} and \ref{u3} and Example \ref{u2}, we conclude the following. 
\begin{theorem}
Let $w$ be a binary word and $l(w)$ be the length of the run sequence of $w$. If all  elements of  $\mathtt{BR}(w)$ are rich, then   $2\leq l(w)\leq 8$.
\end{theorem}
We now consider the following example of a binary word $v$ with $3\leq l(v)\leq 8$ such that 
 there exists an element in  $\mathtt{BR}(v)$ that is not rich.
\begin{example}\label{u4}
   For $3\leq l(v)\leq 8$, consider \[v=\left \{\begin{array}{cc}
   a^2b^3a^3(ba)^i  &  \text{for $i=\frac{l(v)-3}{2}$ and $l(v)$ odd,}\\
   a^2b^3a^3(ba)^ib  & \text{for $i=\frac{l(v)-4}{2}$ and $l(v)$ even.}\end{array} \right.\]

 and
    \[v'=\left \{\begin{array}{cc}
  (ba)^i a^2 b^2 a b a^2  &  \text{for $i=\frac{l(v)-3}{2}$ and $l(v)$ odd,}\\
  (ba)^i b a^2 b^2 a b a^2  & \text{for $i=\frac{l(v)-4}{2}$ and $l(v)$ even.}
  \end{array} \right.\] 
  
  It can observed that
  $v' \in \mathtt{BR}(v)$. Note that  by Lemma \ref{rich}, $v'$ is not rich as $a^2 b^2 a b a^2$ is not rich. 
\end{example}
We conclude from Examples \ref{u2} and \ref{u4} that  all  elements of $\mathtt{BR}(w)$ may or may not be rich for a binary word $w$ such that $3\leq l(w)\leq 8$. Now, we find the structure of binary words $w$ with $3\leq l(w)\leq 8$ such that the block reversal of $w$ contains only rich words. 


We recall the following from Mahalingam et al. (\cite{blore}).
\begin{lemma}\label{f1}\cite{blore}
Let $w\in \Sigma^+$,  $(\mathtt{BR}(w))^R = \mathtt{BR}(w^R)$.
\end{lemma}
 
 We conclude the following from Lemmas \ref{rich} and \ref{f1}.
 
 \begin{remark}\label{o1}
 All  elements of $\mathtt{BR}(w)$ are rich iff all  elements of $\mathtt{BR}(w^R)$ are rich.
 \end{remark}
 We first study the case when the length of the run sequence of the word is equal to $8$.
 
\begin{proposition}\label{g8}
 Let $w\in \{a, b\}^*$ and $l(w)=8$. Then, all  elements of $\mathtt{BR}(w)$ are  rich iff $w=abababab$.
\end{proposition}
\begin{proof}
 Let $w$ be a binary word with $l(w)=8$ such that all  elements of $\mathtt{BR}(w)$ are rich. Consider $w=a^{n_1}b^{n_2}a^{n_3}b^{n_4}a^{n_5}b^{n_6}a^{n_7}b^{n_8}$ to be the run-length encoding of $w$ where $n_i\geq 1$ for all $i$. Let $$w'=b^{n_8+n_6} a^{n_7+n_5} b a  b^{n_4-1+n_2} a^{n_3-1+n_1} \in  \mathtt{BR}(w).$$
Then, $w'$ is rich. If $n_4 \geq 2$ or $n_2\geq 2$, then since $v=b^{2} a^{n_7+n_5} b a  b^{2}\in Fac(w')$ and $v$ contains exactly two occurrences of $b^2$, 
by Theorem \ref{tglen},   $b^{2} a^{n_7+n_5} b a  b^{2}$ is a palindrome  which is a contradiction. Thus, $n_2=n_4=1$. Now, by Remark \ref{o1}, we get, $n_7=n_5=1$. Thus,  $w=a^{n_1}ba^{n_3}bab^{n_6}ab^{n_8}$.
 Now, consider $$w''= b^{n_8} a^{2} b^{n_6+1} a b a^{n_3-1+n_1} \in \mathtt{BR}(w).$$ 
 Then, $w''$ is rich. If $n_1\geq 2$ or $n_3\geq 2$, then since $v'=a^{2} b^{n_6+1} a b a^{2} \in Fac(w'')$ and $v'$ contains exactly two occurrences of $a^2$, by Theorem \ref{tglen},  $ a^{2} b^{n_6+1} a b a^{2}$ is a palindrome which is a contradiction. Thus, $n_1=n_3=1$.  Now, by Remark \ref{o1}, we get, $n_8=n_6=1$. Thus,  $w=abababab$.

 The converse follows from  Theorem  \ref{7ric}.
\end{proof}
 We conclude the following from Proposition \ref{g8}.
 \begin{remark}
Let $w$ be a binary word such that  $l(w)=8$ and $|w|>8$. Then, there exists an element in $\mathtt{BR}(w)$ that is not  rich.
 \end{remark}
 We now consider the case when the length of the run sequence of the word is $7$. For a binary word $w$, if $l(w)=7$, then  $|w|\geq 7$. If $|w|=7$, it is well known (\cite{pcofw}) that all elements of $\mathtt{BR}(w)$ are rich. We consider the case when $|w|>7$ in the following.

\begin{proposition}\label{g7}
Let $w$ be a binary word with $|w|>7$ and $l(w)=7$. Then, all elements of $\mathtt{BR}(w)$ are rich iff $w$ is $ababab^2a$, $abab^2aba$ or $ab^2ababa$.
\end{proposition}
\begin{proof}
 Let $w$ be a binary word with $l(w)=7$ such that all  elements of $\mathtt{BR}(w)$ are rich.  Consider $w=a^{n_1}b^{n_2}a^{n_3}b^{n_4}a^{n_5}b^{n_6}a^{n_7}$ to be the run-length encoding of $w$ where $n_i\geq 1$ for all $i$. Let $$\alpha = a^{n_7-1+n_5} b^{n_6} a b^{n_4+n_2} a^{n_3+n_1} \in \mathtt{BR}(w)$$ and $$\beta =a^{n_7-1+n_5} b^{n_6} a b^{n_4} a^{n_3+n_1} b^{n_2} \in \mathtt{BR}(w) .$$ Then, $\alpha,\; \beta$ are rich.
 If $n_7 \geq 2$ or $n_5 \geq 2$, then by Theorem \ref{tglen},  $a^{2} b^{n_6} a b^{n_4+n_2} a^{2}$ and $a^{2} b^{n_6} a b^{n_4} a^{2}$ are  palindromic factors of $\alpha$ and $\beta $, respectively. This implies $n_6=n_4+n_2$ and  $n_6=n_4$, which is a contradiction to the fact that $n_2 \geq 1$. Thus, $n_7 = n_5 =1$. \\
 Since, $n_7=n_5=1$, by Remark \ref{o1}, we get, $n_1=n_3=1$.  Thus,  $w=a b^{n_2}a b^{n_4}a b^{n_6} a$.

 We now show that $n_6\leq2$. Consider $\gamma = b^{n_6-1} aa b a b^{n_4+n_2} a \in \mathtt{BR}(w)$. Then, $\gamma$ is rich. If $n_6\geq 3$, then by Theorem \ref{tglen},  $b^{2} a^2 b a b^{2}$ is a palindromic factor of $\gamma$ which is a contradiction. Thus, $n_6\leq 2$. We have the following cases:
 \begin{enumerate}
     \item $n_6=2: $ If $n_4\geq 2$ or $n_2\geq 2$, then $b^2 a^2 ba b^{n_4-1+n_2} a \in \mathtt{BR}(w)$ is not rich  which is a contradiction. Thus, in this case, $n_4=n_2=1$. We have, $w=a b a b a b^{2} a$.
     \item $n_6=1 :$ Here, $w=a b^{n_2}a b^{n_4}a b  a$. If $n_4\geq 2$ and $n_2\geq 2$, $b^{n_4} a b a^2 b^{n_2} a \in \mathtt{BR}(w)$ is not rich which is a contradiction. So, either $n_2=1$ or $n_4=1$. Note that if $n_2=n_4=1$, then $|w|=7$, which is a contradiction. We are left with the following cases:
     \begin{itemize}
         \item $n_2=1$ and $n_4 \geq 2$ : Here,   $w=a b a b^{n_4}a b  a$. If $n_4\geq 3$, $b^{n_4-1} a b a^2 b^2 a \in \mathtt{BR}(w)$ is not rich, a contradiction. Thus, $n_4 =2$ and  $w=a b a b^{2}a b  a$. 
         \item $n_4=1$ and $n_2\geq 2$ :  Here,   $w=a b^{n_2} a ba b  a$. If $n_2\geq 3$, $ a b^2 a b a^2 b^{n_2-1}  \in \mathtt{BR}(w)$ is not rich, a contradiction. Thus, $n_2=2$ and $w=a b^2 a ba b  a$. 
     \end{itemize} 
 \end{enumerate}
The converse follows from Theorem \ref{7ric}.
\end{proof}

 We conclude the following from Proposition \ref{g7}.
 \begin{remark}
Let $w$ be a binary word such that  $l(w)=7$ and $|w|>8$. Then, there exists an element in $\mathtt{BR}(w)$ that is not  rich.
 \end{remark}

 We now consider the case when the length of the run sequence of the word is $6$.  
We  need the following:


 \begin{remark}\label{rem2}
We consider the block reversal of the following words: 
 \begin{enumerate}
     \item  Let $w=a^2 b a b a b^{n_6}$  for $n_6<4$. 
     \begin{itemize}
         \item If $n_6=1$ or $2$, then as   $|w|\leq 8$, by Theorem \ref{7ric}, all  elements of $\mathtt{BR}(w)$ are rich.
        
         \item If $n_6=3$, then  $b a b^2 a^2 bab \in \mathtt{BR}(w)$ is not rich which implies that 
          not all  elements of $\mathtt{BR}(w)$ are rich.
     \end{itemize}
      \item  Let $w=a^3 b a b a b^{n_6}$  for $n_6<4$.  
     
     \begin{itemize}
         \item If $n_6=1$, then by Theorem \ref{7ric}, all  elements of $\mathtt{BR}(w)$ are rich.
         \item  If $n_6=2$, then  $a b a b^2 a^2 b a \in \mathtt{BR}(w)$ is not rich which implies that 
          not all  elements of $\mathtt{BR}(w)$ are rich.
         \item If $n_6=3$, then  $ba b^2 a^3 bab \in \mathtt{BR}(w)$ is not rich which implies that 
          not all  elements of $\mathtt{BR}(w)$ are rich.
     \end{itemize}
       \item  Let $w=a^{n_1} b a b a b$  for $n_1\geq 1$.  
       We show that  all elements of $\mathtt{BR}(w)$ are rich.
    Let $w=D_1 D_2 D_3 \cdots D_k$ where $k \geq 2$ and each $D_i \in \Sigma^+$. Then,     either  $D_1 = a^j$, $D_2=a^{n_1-j}x$ and $D_3 \cdots D_k = y$ where $x, y \in \Sigma^*$, $x y =babab$ and $1 \leq j <n_1$
       or
       $D_1 = a^{n_1} x'$ and $D_2 D_3 \cdots D_k=y'$ where  $x' \in \Sigma^*, y' \in \Sigma^+$ and $x' y' = babab$. Then, for distinct elements of $\mathtt{BR}(w)$, we can divide $w$ in at most  seven non-empty blocks.
\begin{itemize}
\item  When we divide $w$ in two non-empty blocks, then $\mathtt{BR}(w)$ contains the following:\\
       Let, $A_2 = \{ b a^{n_1} baba, ab a^{n_1} bab, bab a^{n_1} ba, abab a^{n_1} b, babab a^{n_1}, a^{n_1-i}babab a^i~ |~ 1 \leq i \leq n_1-1 \}.$ We can observe that each element of $A_2$ is rich.
       \item   When we divide $w$ in three non-empty blocks, then $\mathtt{BR}(w)$ contains the following:\\
       Let,  $A_3 = \{ ba a^{n_1} bab, bba a^{n_1}ba, baba a^{n_1}b, bbaba a^{n_1}, b a^{n_1-i} baba a^{i}, abba^{n_1}ba, ababa^{n_1}b, ab bab a^{n_1},\\ ab a^{n_1-i} bab a^{i}, babba a^{n_1}, bab a^{n_1-i} ba a^{i}, abab b a^{n_1}, abab a^{n_1-i} b a^{i}, babab a^{n_1}      ~|~ 1 \leq i \leq n_1-1  \}$. We can observe that each element of $A_3$ is rich.
       \item    When we divide $w$ in four non-empty blocks, then $\mathtt{BR}(w)$ contains the following:\\
        Let, $A_4 = \{ bab a^{n_1} ba, baab a^{n_1} b, babab a^{n_1}, ba a^{n_1-i} bab a^{i}, bbaa a^{n_1} b, bb aba a^{n_1}, bba a^{n_1-i} ba a^{i}, baba a^{n_1-i} b a^{i},\\ b baba a^{n_1}, abba a^{n_1} b, abbba a^{n_1}, abb a^{n_1-i} ba a^{i}, ababb a^{n_1}, abab a^{n_1-i} b a^{i}, abb ab a^{n_1}, babba a^{n_1}, ababb a^{n_1}   ~|~ 1 \leq i \leq n_1-1 \}$. We can observe that  each element of $A_4$ is rich.
        \item When we divide $w$ in five non-empty blocks, then $\mathtt{BR}(w)$ contains the following:\\
        Let, $A_5 =\{ baba a^{n_1} b, babba a^{n_1}, bab a^{n_1-i} ba a^{i}, baab b a^{n_1}, baab a^{n_1-i} b a^{i}, ba bab a^{n_1}, bbaa b a^{n_1}, bb aa a^{n_1-i} b a^{i},\\ bbaba a^{n_1}, abbab a^{n_1}, abba a^{n_1-i} b a^{i}, ababb a^{n_1}, a bbb a a^{n_1}     ~|~ 1 \leq i \leq n_1-1   \}$.  We can observe that  each element of $A_5$ is rich.
        \item  When we divide $w$ in six non-empty blocks, then $\mathtt{BR}(w)$ contains the following:\\
         Let, $A_6 = \{  babab a^{n_1}, baba a^{n_1-i} b a^{i}, babba a^{n_1}, baabb a^{n_1}, bbaab a^{n_1}, abb ab a^{n_1}    ~|~ 1 \leq i \leq n_1-1\}$. We can observe that each element of $A_6$ is rich.
         \item When we divide $w$ in seven non-empty blocks, then $\mathtt{BR}(w)$ contains the following:\\
          Let, $A_7=\{ babab a^{n_1} \}$. Clearly, $babab a^{n_1}$ is rich.
          Thus, each element of $A_7$ is rich. Also, as   $babab a^{n_1}$ is rich, $w=a^{n_1} babab $ is rich. 
\end{itemize}
          Therefore, as $\mathtt{BR}(w)= \{w\} \cup A_2 \cup A_3 \cup A_4 \cup A_5 \cup A_6 \cup A_7$, each element of $\mathtt{BR}(w)$ is rich.

          In a similar fashion one can also show that each element of $\mathtt{BR}(ababab^{n_1})$ is also rich.

 \end{enumerate}
 \end{remark}

We now have the following:
\begin{proposition}\label{g6}
Let $w$ be a binary word with $|w|>7$ and $l(w)=6$. Then, all  elements of $\mathtt{BR}(w)$ are rich iff $w\in \{u, (u^c)^R\;|\; u\in  T\}$ where $$T=\{a b^2 a b a^2 b, a b a b^2 a^2 b, a b a b a^2 b^2, a^2 b a b^2 a b, a^2 b a b a b^2, a b a^2 b^2 a b, a^{n_1} b a b a b\; |\;n_1\geq 3\}.$$ 

\end{proposition}
\begin{proof}
Let $w$ be a binary word  with $|w|>7$ and $l(w)=6$ such that all  elements of $\mathtt{BR}(w)$ are rich. Consider $w=a^{n_1}b^{n_2}a^{n_3}b^{n_4}a^{n_5}b^{n_6}$ to be the run-length encoding of $w$ where $n_i\geq 1$ for all $i$. If $n_5\geq 3$, then by Theorem \ref{tglen}, $ b^{n_6-1} a^{n_5-1} b a b^{n_4+n_2} a^{n_3+n_1} \in \mathtt{BR}(w)$ is not rich. This implies, $n_5\leq 2$.
By Remark \ref{o1}, we get, $n_2\leq 2$. We have the following cases:
\begin{itemize}
    \item $n_5 = 2 : $  If $n_1 \geq 2$ or $n_3\geq 2$, then by Theorem \ref{tglen},  $ b^{n_6-1} a^{2} b a b^{n_4+n_2} a^{n_3-1+n_1} \in \mathtt{BR}(w)$ is not rich which is a contradiction. Thus,  $n_1 =n_3= 1$.  Thus, in this case, we get,
      $w = a b^{n_2} a b^{n_4} a^2 b^{n_6}$ where $n_2\leq 2$.
      We have the following cases:  
    \begin{itemize}
            \item $n_2=2 : $  Then, $w^R = b^{n_6} a^{2} b^{n_4} a b^{2} a$. By Remark  \ref{o1}, all  elements of $\mathtt{BR}(w^R)$ are rich. Here, $n_4=n_6=1$, otherwise $a b^2 a b a^2 b^{n_4-1+n_6} \in \mathtt{BR}(w^R)$ is not rich. Thus, $w = a b^{2} a b a^2 b$.
            \item $n_2 = 1 : $ Then, $w = a b a b^{n_4} a^2 b^{n_6}$. We have the following cases:
                \begin{itemize}
              \item $n_4 \geq 2$ : If $n_6 \geq 2$, then by Theorem \ref{tglen}, $b^{n_6} a^2 b a b^{n_4} a \in \mathtt{BR}(w)$ is not rich which is a contradiction. Otherwise,  $n_6 =1$ and $w = a b a b^{n_4} a^2 b$.  If $n_4 \geq 3$, then by Theorem \ref{tglen}, $b^{n_4-1} a^2 b a b^2 a \in \mathtt{BR}(w)$ is not rich which is a contradiction. Thus, $n_4 =2$ and  $w = a b a  b^{2} a^2 b$.   
              \item    $n_4=1$ :   If $n_6 \geq 3$, then by Theorem \ref{tglen}, $ b^{n_6-1} a^2 b a b^2 a \in \mathtt{BR}(w)$ is not rich which is a contradiction.  Thus, $n_6\leq 2$. If $n_6=1$, then $|w|=7$, thus, $n_6=2$ and  $w = a b a b a^2 b^{2}$.
                \end{itemize}
        \end{itemize}
    \item $n_5 = 1 : $ Here, $w=a^{n_1}b^{n_2}a^{n_3}b^{n_4} a b^{n_6}$. We have the following cases:
        \begin{itemize}
        \item  $n_1 \geq 2$ :  If $n_3\geq 2$, then by Theorem \ref{tglen}, $ a^{n_3} b^{n_4} a b^{n_6} a^{n_1} b^{n_2} \in \mathtt{BR}(w)$  for $n_4 \neq n_6$ ($ a^{n_3} b^{n_4} a b^{n_6+n_2} a^{n_1} \in \mathtt{BR}(w)$ for $n_4 = n_6$, respectively) is not rich which is a contradiction. So, $n_3=1$ and $w= a^{n_1}b^{n_2} a b^{n_4}a b^{n_6}$ where $n_2\leq 2$. We have the following cases:

                \begin{itemize}
                    \item $n_2 = 2 : $  Now, $w^R = b^{n_6} a b^{n_4} a b^2 a^{n_1}$. By Remark \ref{o1},  all  elements of $\mathtt{BR}(w^R)$ are rich. If $n_4\geq 2$ or $n_6\geq 2$, then $a^{n_1-1} b^2 a b a^2 b^{n_4-1+n_6} \in \mathtt{BR}(w^R)$
                    is not rich, a contradiction. Thus, $n_4 = n_6=1$. We get, $w= a^{n_1}b^{2} a b a b $.  If $n_1\geq 3$, then by Theorem \ref{tglen}, $b a^{n_1-1} b^2 a b a^2 \in \mathtt{BR}(w)$ is not rich which is a contradiction.  Thus, $n_1=2$ and $w= a^2 b^{2} a b a b$.
                    \item $n_2 = 1 : $ Then, $w= a^{n_1} b a b^{n_4}a b^{n_6}$. We have the following cases:     \begin{itemize}
                 \item  $n_4\geq 2$ : If $n_6 \geq 2$,   then by Theorem \ref{tglen},   $b^{n_6} a^{n_1} b a b^{n_4}a \in \mathtt{BR}(w)$  is not rich which is a contradiction. Otherwise, $n_6=1$. Here, $n_1=2$,  otherwise, by Theorem \ref{tglen}, $b a^{n_1-1} b a b^{n_4} a^2 \in \mathtt{BR}(w)$  is not rich. Thus, $w= a^{2} b a b^{n_4} a b$. Also, $n_4=2$, otherwise, by Theorem \ref{tglen}, $ab b^{n_4-2} a^2 b a b^2 \in \mathtt{BR}(w)$ is not rich. Thus, $w= a^{2} b a b^{2} a b$.
                 
               \item  $n_4 = 1 : $   If $n_6\geq 4$, then by Theorem \ref{tglen}, $b^{n_6-2} a^{n_1} b a b a b^2\in \mathtt{BR}(w)$ is not rich which is a contradiction. Thus, $w=a^{n_1} b a b a b^{n_6}$
                 for $1\leq n_6\leq 3$. In this case, if $n_1\geq 4$ and $2\leq n_6 \leq 3$, then by Theorem \ref{tglen}, $a^{n_1-2} b a b a b^{n_6} a^2 \in \mathtt{BR}(w)$ is not rich. We are left with the words
                 $a^2 b a b a b^{n_6}$, $a^3 b a b a b^{n_6}$ for $n_6<4$ and $a^{n_1} b a b a b$ for $n_1\geq 1$. By Remark \ref{rem2}, one can conclude that if all  elements of $\mathtt{BR}(w)$ are rich, then $w$ is $a^{n_1} b a b a b$ or $a^{2} b a b a b^{2}$ where $n_1 \geq 3$.

                        \end{itemize}
                \end{itemize}
            
            \item $n_1=1$
             : Similar to the case $n_1=2$, one can prove that if all  elements of $\mathtt{BR}(w)$ are rich, then $w$ is of one of the following forms: 
             $a b^2 a^2 b a b,\; a b a^2 b^2 a b, \; a b a^2 b a b^2,\; a b a b a b^{n_6} $, where $n_6 \geq 3$.

        \end{itemize}

    \end{itemize}
      The converse follows from Theorem \ref{7ric} and  Remark \ref{rem2}.
    
\end{proof}

We now consider the case when the length of the run sequence of the word is $5$.

\begin{proposition}\label{g5}
Let $w$ be a binary word with $|w|>7$ and $l(w)=5$. Then, all  elements of $\mathtt{BR}(w)$ are rich iff $w$ is   $a^{n_1}ba^{n_3}ba^{n_5}$, $ab^{n_2}ab^{n_4}a^2$,  $a^2b^{n_2}ab^{n_4}a$, 
$ab^{n_2}a^2b^{n_4}a$ where $n_2+n_4 =4$ for   $n_i\geq 1$.
\end{proposition}
\begin{proof}
Let $w$ be a binary word with $|w|>7$  and $l(w)=5$ such that all  elements of $\mathtt{BR}(w)$ are rich. Consider $w=a^{n_1}b^{n_2}a^{n_3}b^{n_4}a^{n_5}$ to be the run-length encoding of $w$ where $n_i\geq 1$ for all $i$.
If  $n_2 = n_4 = 1$, then by Theorem \ref{tglen}, all  elements of $\mathtt{BR}(w)$ are rich. Now, consider $n_2 + n_4 \geq 3$.  If $n_5 \geq 3$, then by Theorem \ref{tglen}, $a^{n_5-1} b a b^{n_4-1+n_2} a^{n_3+n_1} \in  \mathtt{BR}(w)$ is not rich which is a contradiction. Thus,  $n_5\leq 2$. By Remark \ref{o1}, we get,   $n_1\leq 2$.  We are left with the following cases: 
\begin{itemize}
    \item   $n_5=2 : $ Then, $w=a^{n_1}b^{n_2}a^{n_3}b^{n_4}a^{2}$ where $n_2 + n_4 \geq 3$ and $n_1\leq 2$. If $n_3\geq 2$ or $n_1= 2$, then by Theorem \ref{tglen}, $a^{2} b a b^{n_4-1+n_2} a^{n_3-1+n_1} \in \mathtt{BR}(w)$ is not rich which is a contradiction. Thus, $n_3=n_1=1$ and in this case, $w=a b^{n_2} a b^{n_4} a^{2}$ where $n_2 + n_4 \geq 3$.  As $|w|>7$, we get, $n_2+n_4\geq 4$. If $n_2+n_4= 4$, then by Theorem \ref{7ric}, all  elements of $\mathtt{BR}(w)$ are rich. We now consider the case when $n_2 + n_4 \geq 5$. We have the following cases:
    \begin{itemize}
      \item  $n_2=i$ for $1\leq i\leq 3$:  Then, $n_4 \geq 5-i$ and $w=a b^i a b^{n_4} a^{2}$. By Theorem \ref{tglen}, we get,  $b^{n_4-3+i} a^2 b a b^2 a \in \mathtt{BR}(w)$ is not rich which is a contradiction.
        \item $n_2 \geq 4 : $ Then, $n_4 \geq 1$ and $w=a b^{n_2} a b^{n_4} a^{2}$. If $n_4 = 1$, then $b^{n_2-2} a b a^2 b^2 a \in \mathtt{BR}(w)$ is not rich which is a contradiction. If $n_4 \geq 2$, then $b^{n_4} a^2 b a b^{n_2-1} a \in \mathtt{BR}(w)$ is not rich which is a contradiction.

    \end{itemize}
  Hence, if $n_5=2$ and  all  elements of $\mathtt{BR}(w)$ are rich, then $w=a b^{n_2} a b^{n_4} a^{2}$ where $n_2+n_4=4$.  
  \item
  $n_5= 1 $: If $n_1=2$, then by  Remark \ref{o1}, we get from the case $n_5=2$, if  all  elements of $\mathtt{BR}(w)$ are rich, then $w= a^2 b^{n_2} a b^{n_4} a$ 
  where $n_2+n_4=4$. Otherwise, $n_1=1$. Then, $w=ab^{n_2}a^{n_3}b^{n_4}a$ where $n_2+n_4 \geq 3$. We have the following cases:
   \begin{itemize}
             \item  $n_3 \geq 3$ :  If $n_2=n_4$, then since $n_2+n_4 \geq 3$, we get both $n_2,\;n_4 \geq 2$. By Theorem \ref{tglen}, $a^{n_3-1} b^{n_4} a b a^{2} b^{n_2-1} \in \mathtt{BR}(w)$ is not rich which is a contradiction. Thus, $n_2 \neq n_4$. By Theorem \ref{tglen}, $a^{n_3-1} b^{n_4} a b^{n_2} a^{2} \in \mathtt{BR}(w)$ is not rich which is a contradiction.

                \item $n_3\leq 2$ : Then, $w = a b^{n_2} a^{n_3} b^{n_4} a$ where $n_2+n_4\geq 3$. As $|w|>7$, we get,  $n_2+n_4\geq 4$. If $n_2+n_4 = 4$, then by Theorem \ref{7ric}, all  elements of $\mathtt{BR}(w)$ are rich. Thus,  $w = a b^{n_2} a^{n_3} b^{n_4} a$ where $n_2+n_4 = 4$.
                Now, we consider the case when $n_2+n_4\geq 5$.  We have the following cases:
                  \begin{itemize}
                      \item  $n_2=i$ for $1\leq i\leq 3$:  Then, $n_4 \geq 5-i$ and $w=a b^i a^{n_3} b^{n_4} a$. By Theorem \ref{tglen}, we get,  $b^{n_4-3+i} a b a^2 b^2 a \in \mathtt{BR}(w)$ for $n_3=2$ ( $b^{n_4-3+i} a^2 b a b^2  \in \mathtt{BR}(w)$ for $n_3=1$, respectively) is not rich which is a contradiction.
                  
                      \item  $n_2 \geq 4 : $ Then, $n_4 \geq 1$ and $w = a b^{n_2} a^{n_3} b^{n_4} a$. If $n_4=1$, then  by Theorem \ref{tglen}, $b^{n_2-2} a^2 b a b^2 a \in \mathtt{BR}(w)$ for $n_3 =2$ ( $b^{n_2-2} a b a^2 b^2  \in \mathtt{BR}(w)$ for $n_3 =1$, respectively ) is not rich which is a contradiction. Otherwise, $n_4 \geq 2$.  By Theorem \ref{tglen}, $b^{n_4} a b a^2 b^{n_2-1} a \in \mathtt{BR}(w)$ for $n_3=2$ ( $b^{n_4} a^2 b a b^{n_2-1}  \in \mathtt{BR}(w)$ for $n_3=1$, respectively) is not rich which is a contradiction. 
                  \end{itemize}

        \end{itemize}
    
    \end{itemize}
    The converse follows from Theorems  \ref{tglen} and \ref{7ric}.

\end{proof}

 We  consider the case when the length of the run sequence of the word is $4$ and the length of the word is greater than $7$. We have the following result.
\begin{proposition}\label{g4}
Let $w$ be a binary word with $|w|>7$ and $l(w)=4$. Then, all  elements of $\mathtt{BR}(w)$ are rich iff $w$ is of the form $ab^{n_2}ab^{n_4}$ or $a^{n_1}b a^{n_3}b$ or $w\in S$ where 
\[S=\big\{
a^{n_1}b^{n_2}a^{n_3}b^{n_4}| (n_2,n_4), (n_1,n_3) \in \{(3,1),(2,2),(1,3)\}\big\} \]
and $n_i\geq 1$.

\end{proposition}

\begin{proof}
Let $w$ be a binary word with $l(w)=4$ such that all  elements of $\mathtt{BR}(w)$ are rich. Consider $w=a^{n_1}b^{n_2}a^{n_3}b^{n_4}$ to be the run-length encoding of $w$ where $n_i\geq 1$ for all $i$.
If $n_1=n_3=1$ or $n_2= n_4 = 1$, then all  elements of $\mathtt{BR}(w)$ are rich. Now, we consider the case $n_1+n_3 \geq 3$ and $n_2 + n_4 \geq 3$. If $n_3 \geq 4$, then for $j \neq r$ and $i+j=n_4$, and $r+s = n_2$, consider  $\alpha = b^i a^{n_3-2}  b^j a  b^r a^{1+n_1} b^s \in \mathtt{BR}(w)$.  By Theorem \ref{tglen}, $\alpha$ is not rich which is a contradiction. Thus, $n_3 \leq 3$. We have the following cases:
\begin{itemize}
    \item $n_3=3$ : Then, for  $j' \neq r'$ and $i'+j'=n_4$, and $r'+s' = n_2$, consider  $\beta = b^{i'} a^{2}  b^{j'} a  b^{r'} a^{n_1} b^{s'} \in \mathtt{BR}(w)$. If $n_1 \geq 2$, then by Theorem \ref{tglen}, $\beta$ is not rich which is a contradiction. Thus, if $n_3=3$, then $n_1=1$. 
    
    \item $n_3=2: $ Then, for $j'' \neq  n_2$ and $i''+j''=n_4$, consider  $\gamma = b^{i''} a^{2}  b^{j''} a  b^{n_2} a^{n_1-1} \in \mathtt{BR}(w) $. If $n_1 \geq 3$, $\gamma $ is not rich which is a contradiction. Thus, if $n_3=2$, then $n_1\leq 2$.
    
    \item  $n_3=1 :$ Then, by Theorem \ref{tglen}, for $n_1 \geq 4$, $a^{n_1-2}b^{n_2}a b^{n_4} a^2 \in \mathtt{BR}(w)$ is not rich when $n_2 \neq n_4$ and  $ b a^{n_1-2}b^{n_2}a b^{n_4-1} a^2 \in \mathtt{BR}(w)$ is not rich when $n_2 = n_4$. Thus, if $n_3=1$, then $n_1\leq 3$.
    
\end{itemize}

Now, $w^R = b^{n_4} a^{n_3} b^{n_2} a^{n_1}$. Then from Remark \ref{o1},
one can similarly deduce that $n_2\leq 3$ and we also conclude  the following:
\begin{itemize}
    \item If $n_2=3$, then $n_4=1$.
    \item If $n_2=2$, then $n_4\leq 2$. 
\item If $n_2=1$, then $n_4\leq 3$.

\end{itemize}

Hence, as $|w|>7$, we get, $w\in S$ where 
\[S=\big\{
a^{n_1}b^{n_2}a^{n_3}b^{n_4}| (n_2,n_4), (n_1,n_3) \in \{(3,1),(2,2),(1,3)\}\big\} \]
for $n_i\geq 1$ and $n_1+n_2+n_3+n_4>7$.\\
The converse follows from  Theorems \ref{tglen} and \ref{7ric}. 
\end{proof}

We  consider the case when the length of the run sequence of the word is $3$ and the length of the word is greater than $7$. We need the following result.

\begin{lemma}\label{l1}
Let  $w=a^{n_1} b^{n_2} a$ or $a b^{n_2} a^{n_1}$ where $n_1\geq 1$ and $n_2\in \{3, 4\}$. Then, all  elements of $\mathtt{BR}(w)$ are rich.
\end{lemma}
\begin{proof}
We prove the result for  $w=a^{n_1} b^{n_2} a$. The proof for $w=a b^{n_2} a^{n_1}$ is similar. Let  $w=a^{n_1} b^{n_2} a$ where $n_1\geq 1$ and $n_2\in \{3, 4\}$. Consider $w' \in \mathtt{BR}(w)$. Then, one can observe that  $2\leq l(w')\leq 6$. If $l(w')\leq 4$, then from Proposition \ref{runlen}, $w'$ is rich. We are left with the following cases:
\begin{itemize}
    \item $l(w')=5 :$ Then, $w'$ is either $a b^{i_1} a^{i_2} b^{i_3} a^{i_4}$ or $b^{j_1} a b^{j_2} a^{n_1} b^{j_3}$ where $i_1+i_3=n_2$, $i_2+i_4 = n_1$ and $j_1+j_2+j_3=n_2$. By Theorem \ref{tglen}, $w'$ is rich.
    \item $l(w')=6 :$ Then, $w'= b^{k_1} a b^{k_2} a^{k_3} b^{k_4} a^{k_5} $ where $k_1+k_2+k_4=n_2$ and $k_3+k_5=n_1$. By Theorem \ref{tglen}, $w'$ is rich.
\end{itemize}
\end{proof}
We have the following:
\begin{proposition}\label{g3}
Let $w$ be a binary word with $|w|>7$ and $l(w)=3$. Then, all  elements of $\mathtt{BR}(w)$ are rich iff $w$ is of one of the following forms:
\begin{itemize}
\item $a^2 b^4 a^2$, $ab^{m_2}a$,   $a^{m_1}ba^{m_3}$, $a^{m_1}b^2a^{m_3}$.
\item $a b^{n_2} a^{n_3}, a^{n_1} b^{n_2} a $  where $n_2 \in \{ 3, 4\}$ and $ n_1, n_3 \geq 3$.

\end{itemize}
where $m_i\geq 1$.
\end{proposition}

\begin{proof}
Let $w$ be a binary word with $l(w)=3$ such that all  elements of $\mathtt{BR}(w)$ are rich. Consider $w=a^{n_1}b^{n_2}a^{n_3}$ to be the run-length encoding of $w$ where $n_i\geq 1$ for all $i$.
If $n_1=n_3=1$ or $n_2 \in \{1, 2\}$, then all  elements of $\mathtt{BR}(w)$ are rich. Here, $w=ab^{n_2}a$ or  $a^{n_1}ba^{n_3}$ or $a^{n_1}b^2a^{n_3}$. 

Now, we consider $n_1+n_3 \geq 3$ and $n_2 \geq 3$.
 Let $n_2\geq 5$ and for $j \neq r$, $i+j=n_3$ and $r+s=n_1$, consider $\alpha = a^i b^2 a^j b a^r b^{n_2-3} a^s\in \mathtt{BR}(w)$. Since, $b^2 a^j b a^r b^{2}$ is not a palindrome, by Theorem \ref{tglen}, $\alpha$ is not rich which is a contradiction. Thus, $n_2\leq 4$. So, $n_2 \in \{ 3, 4\}$. We have the following cases:
 \begin{itemize}
     \item $n_3\geq 3$ : Then, consider $\beta = a^{n_3-1} b a b^2 a^{n_1} b^{n_2-3} \in \mathtt{BR}(w)$. If $n_1\geq 2$, then by Theorem \ref{tglen}, $a^{2} b a b^2 a^{2}$ is a palindrome, which is a contradiction. Thus, if $n_3 \geq 3$, then $n_1=1$. Here, $w= a b^{n_2} a^{n_3}$  where $n_2 \in \{ 3, 4\}$ and $ n_3 \geq 3$.
     \item $n_3=2$ :  Then, consider $\beta' = a^2 b a b^{n_2-1} a^{n_1-1}\in \mathtt{BR}(w)$. If $n_1\geq 3$, then $\beta'$ is not rich. Thus, if $n_3=2$, then $n_1 \leq 2$. Thus, as $|w|\geq 8$, $w= a^2 b^4 a^2$.
     \item $n_3=1$ : Then by Lemma \ref{l1}, all  elements of $\mathtt{BR}(w)$ are rich.
     
 \end{itemize}

The converse follows from  Theorems \ref{tglen} and \ref{7ric} and Lemma \ref{l1}.
\end{proof}
We conclude the following from Propositions \ref{g8}, \ref{g7}, \ref{g6}, \ref{g5}, \ref{g4} and \ref{g3} for the binary words with $3\leq l(w) \leq 8.$
\begin{theorem}
 Let $w$ be a binary word with $|w|>7$. Then,  all  elements of $\mathtt{BR}(w)$ are  rich for
\begin{enumerate}
    \item  $l(w)=8$  iff $w=abababab$
    
    \item  $l(w)=7$  iff $w$ is $ababab^2a$, $abab^2aba$ or $ab^2ababa$
    
    \item  $l(w)=6$  iff $w\in \{u, (u^c)^R\;|\; u\in  T\}$ where $$T=\{a b^2 a b a^2 b, a b a b^2 a^2 b, a b a b a^2 b^2, a^2 b a b^2 a b, a^2 b a b a b^2, a b a^2 b^2 a b, a^{n_1} b a b a b\; |\;n_1\geq 3\}$$ 
    
    \item  $l(w)=5$  iff $w$  is  $a^{n_1}ba^{n_3}ba^{n_5}$,  $ab^{n_2}ab^{n_4}a^2$,  $a^2b^{n_2}ab^{n_4}a$, 
$ab^{n_2}a^2b^{n_4}a$ where $n_2+n_4 =4$

    \item $l(w)=4$  iff $w$ is $ab^{n_2}ab^{n_4}$ or $a^{n_1}b a^{n_3}b$ or $w\in S$ where 
\[S=\big\{
a^{n_1}b^{n_2}a^{n_3}b^{n_4}| (n_2,n_4), (n_1,n_3) \in \{(3,1),(2,2),(1,3)\}\big\} \]
\item   $l(w)=3$ iff $w$ is of one of the following forms:
\begin{itemize}
\item $a^2 b^4 a^2$, $ab^{n_2}a$,   $a^{n_1}ba^{n_3}$, $a^{n_1}b^2a^{n_3}$
\item $a b^{n_2} a^{n_3}, a^{n_1} b^{n_2} a $  where $n_2 \in \{ 3, 4\}$ and $ n_1, n_3 \geq 3$
\end{itemize}
\end{enumerate}
where $n_i\geq 1$.
\end{theorem}   
 
\section{Conclusions}
In this paper, we have characterized words whose block reversal contains only rich words. We have found necessary and sufficient conditions for a non-binary word such that all elements in its block reversal are rich.
For a binary word, we have showed that the result varies with the length of the run sequence of the word and the structure of the word.
In future, we would like to find an upper bound on the number of elements in the block reversal of the binary word. It would also be interesting to study other combinatorial properties such as counting primitive words, bordered and unbordered words in the block reversal set of a word.

\bibliographystyle{splncs04.bst}
\bibliography{splncs04.bib}
\newpage

\section{Appendix}
Proof of the subcase  $n_5=1$ and $n_1=1$ in Proposition \ref{g6}:
\begin{proof}
     \item  $n_1 = 1 : $ Then, $w= a b^{n_2} a^{n_3} b^{n_4} a b^{n_6}$ where $n_2\leq 2$. We have the following cases:
             \begin{itemize}
                    \item $n_2 = 2 : $ Then, $w^R = b^{n_6} a b^{n_4} a^{n_3} b^{2} a$. If $n_4 \geq 2$ or $n_6 \geq 2$, then by Theorem \ref{tglen}, $b^2 a b a^{n_3+1} b^{n_4-1+n_6} \in \mathtt{BR}(w^R)$ is not rich which is a contradiction. Thus, $n_4 = n_6=1$ and $w= a b^{2} a^{n_3} b a b $. If $n_3\geq 3$, then by Theorem \ref{tglen}, $a^{n_3-1} b a b^3 a^2 \in \mathtt{BR}(w)$ is not rich which is a contradiction.  Thus, as $|w|>7$, $n_3=2$ and $w= a b^{2} a^{2} b a b $.

                    \item $n_2 = 1 : $ Then, $w= a b  a^{n_3} b^{n_4} a b^{n_6}$. 
                    If $n_4 \geq 2$ and $n_6 \geq 2$, then  by Theorem \ref{tglen}, for $n_3 \geq 2$, $b^{n_6} a b a^{n_3} b^{n_4}a \in \mathtt{BR}(w)$ and for $n_3=1$, $b^{n_6} a^2 b a b^{n_4} \in \mathtt{BR}(w)$ are not rich which is a contradiction. Then, we have the following cases:

                        \begin{itemize}
                            \item $n_4 \geq 2$ and $n_6 = 1 : $ Then by Theorem \ref{tglen}, for $n_3 \geq 3$, $b a^{n_3-1} b^{n_4} a b a^2 \in \mathtt{BR}(w)$ is not rich which is a contradiction. For $n_3=2$, $w= a b  a^2 b^{n_4} a b$. If $n_4\geq 3$, then by Theorem \ref{tglen}, $ab b^{n_4-2} a b a^2 b^2 \in \mathtt{BR}(w)$ is not rich which is a contradiction. Thus, $n_4=2$ and  $w= a b  a^2 b^{2} a b$. For $n_3=1$, $w= a b  a b^{n_4} a b$. Since, $|w|>7$, $n_4 \geq 3$. Then, by Theorem \ref{tglen}, $b^2 a^2 b a b^{n_4-1} \in \mathtt{BR}(w)$ is not rich which is a contradiction.

                            \item $n_6 \geq 2$ and $n_4 = 1 : $ Then, $w= a b  a^{n_3} b a b^{n_6}$. Now, by Theorem \ref{tglen}, for $n_3 \geq 3$,  $a^{n_3-1}  b a b^{n_6} a^2 b\in \mathtt{BR}(w)$ is not rich which is a contradiction. For $n_3=2$,  $w= a b  a^{2} b a b^{n_6}$. Then by Theorem \ref{tglen}, for $n_6 \geq 3$, $b^{n_6-1}a^2 b a b^2 a \in \mathtt{BR}(w)$ is not rich which is a contradiction.
                            Thus, $n_6=2$ and $w= a b  a^{2} b a b^{2}$. For $n_3=1$, $w= a b  a b a b^{n_6}$.

                            \item $n_4 = 1 $ and $n_6 = 1 : $ Then, $w= a b  a^{n_3} b a b$. Now, by Theorem \ref{tglen}, for $n_3 \geq 3$, $a^{n_3-1} b a b^2 a^2 \in \mathtt{BR}(w)$ is not rich which is a contradiction. For $n_3\leq 2$, $|w|<8$. Hence, we omit this. 
                        \end{itemize}
                \end{itemize}
                
                            
                            
\end{proof}

\end{document}